\documentclass[11pt, a4paper]{article}

\usepackage{amsmath, amsthm, amssymb, url}
\usepackage[margin=1in]{geometry}
\usepackage{multirow}
\usepackage{tikz, tkz-graph, mathrsfs, enumerate}

\usetikzlibrary{arrows}
\newcommand{\vertex}[3]{\node [vertex] (#1) at (#2, #3 * 1.7) {};}

\newcommand{\arc}[2]{{\draw[-latex] (#1) edge (#2);}}

\newcommand{\Tran}{\mathrm{Tran}}
\newcommand{\Sym}{\mathrm{Sym}}

\newcommand{\id}{\mathrm{id}}
\newcommand{\CA}{\mathrm{CA}}
\newcommand{\ICA}{\mathrm{ICA}}
\newcommand{\Rank}{\mathrm{Rank}}
\newcommand{\Sub}{\mathrm{Sub}}

\theoremstyle{plain}

\newtheorem{corollary}{Corollary}
\newtheorem{lemma}{Lemma}

\newtheorem{theorem}{Theorem}
\newtheorem*{claim*}{Claim}

\theoremstyle{definition}
\newtheorem{definition}{Definition}
\newtheorem{problem}{Problem}
\newtheorem{example}{Example}
\newtheorem{remark}{Remark}

\newtheorem{notation}{Notation}

\usepackage{url}
\urldef{\mailsa}\path|{alonso.castillo-ramirez,m.r.gadouleau}@durham.ac.uk|

\begin{document}

\title{On Finite Monoids of Cellular Automata}
\author{Alonso Castillo-Ramirez\footnote{Email: \texttt{alonso.castillo-ramirez@durham.ac.uk}} \ and Maximilien Gadouleau\footnote{Email: \texttt{m.r.gadouleau@durham.ac.uk}}  \\ \\ 
\small School of Engineering and Computing Sciences, \\ 
\small Durham University, South Road, \\ 
\small Durham, DH1 3LE \\
\small Telephone: +44 (0) 191 33 41729}
\maketitle

\begin{abstract}
For any group $G$ and set $A$, a cellular automaton over $G$ and $A$ is a transformation $\tau : A^G \to A^G$ defined via a finite neighborhood $S \subseteq G$ (called a memory set of $\tau$) and a local function $\mu : A^S \to A$. In this paper, we assume that $G$ and $A$ are both finite and study various algebraic properties of the finite monoid $\CA(G,A)$ consisting of all cellular automata over $G$ and $A$. Let $\ICA(G;A)$ be the group of invertible cellular automata over $G$ and $A$. In the first part, using information on the conjugacy classes of subgroups of $G$, we give a detailed description of the structure of $\ICA(G;A)$ in terms of direct and wreath products. In the second part, we study generating sets of $\CA(G;A)$. In particular, we prove that $\CA(G,A)$ cannot be generated by cellular automata with small memory set, and, when $G$ is finite abelian, we determine the minimal size of a set $V \subseteq \CA(G;A)$ such that $\CA(G;A) = \langle \ICA(G;A) \cup V \rangle$. 
\end{abstract}


\section{Introduction}

Cellular automata (CA), first introduced by John von Neumann as an attempt to design self-reproducing systems, are models of computation with important applications to computer science, physics, and theoretical biology. In recent years, the theory of CA has been greatly enriched with its connections to group theory and topology (see \cite{CSC10} and references therein). One of the goals of this paper is to embark in the new task of exploring CA from the point of view of finite group and semigroup theory.

We review the broad definition of CA that appears in \cite[Sec.~1.4]{CSC10}. Let $G$ be a group and $A$ a set. Denote by $A^G$ the \emph{configuration space}, i.e. the set of all functions of the form $x:G \to A$. For each $g \in G$, let $R_g : G \to G$ be the right multiplication function, i.e. $(h)R_g := hg$, for any $h \in G$. We emphasise that we apply functions on the right, while in \cite{CSC10} functions are applied on the left.   

\begin{definition} \label{def:ca}
Let $G$ be a group and $A$ a set. A \emph{cellular automaton} over $G$ and $A$ is a transformation $\tau : A^G \to A^G$ such that there is a finite subset $S \subseteq G$, called a \emph{memory set} of $\tau$, and a \emph{local function} $\mu : A^S \to A$ satisfying
\[ (g)(x)\tau = (( R_g \circ x  )\vert_{S}) \mu, \ \forall x \in A^G, g \in G.  \]
\end{definition} 

Most of the classical literature on CA focuses on the case when $G=\mathbb{Z}^d$, for $d\geq1$, and $A$ is a finite set (e.g. see survey \cite{Ka05}).

A \emph{semigroup} is a set $M$ equipped with an associative binary operation. If there exists an element $\id \in M$ such that $\id \cdot m = m \cdot \id = m$, for all $m \in M$, the semigroup $M$ is called a \emph{monoid} and $\id$ an \emph{identity} of $M$. Clearly, the identity of a monoid is always unique.  

Let $\CA(G;A)$ be the set of all cellular automata over $G$ and $A$; by \cite[Corollary 1.4.11]{CSC10}, this set equipped with the composition of functions is a monoid. Although results on monoids of CA have appeared in the literature before (see \cite{CRG15,H12,S15}), the algebraic structure of $\CA(G;A)$ remains basically unknown. In particular, the study of $\CA(G;A)$, when $G$ and $A$ are both finite, has been generally disregarded, perhaps because some of the classical questions are trivially answered (e.g. the Garden of Eden theorems become trivial). However, many new questions, typical of finite semigroup theory, arise in this setting.

In this paper, we study various algebraic properties of $\CA(G;A)$ when $G$ and $A$ are both finite. First, in Section \ref{basic}, we introduce notation and review some basic results. In Section \ref{structure}, we study the group $\ICA(G;A)$ consisting of all invertible CA: we show that its structure is linked with the number of conjugacy classes of subgroups of $G$, and we give an explicit decomposition in terms of direct and wreath products. 

In Section \ref{generating}, we study generating sets of $\CA(G;A)$. We prove that $\CA(G;A)$ cannot be generated by CA with small memory sets: if $T$ generates $\CA(G;A)$, then $T$ must contain a cellular automaton with minimal memory set equal to $G$ itself. This result provides a striking contrast with CA over infinite groups. Finally, when $G$ is finite abelian, we find the smallest size of a set $U \subseteq \CA(G;A)$ such that $\ICA(G;A) \cup U$ generate $\CA(G;A)$; this number is known in semigroup theory as the \emph{relative rank} of $\ICA(G;A)$ in $\CA(G;A)$, and it turns out to be related with the number of edges of the subgroup lattice of $G$.      


\section{Basic Results} \label{basic}

For any set $X$, let $\Tran(X)$ and $\Sym(X)$ be the sets of all functions and bijective functions, respectively, of the form $\tau : X \to X$. Equipped with the composition of functions, $\Tran(X)$ is known as the \emph{full transformation monoid} on $X$, while $\Sym(X)$ is the \emph{symmetric group} on $X$. When $X$ is finite and $\vert X \vert = q$, we write $\Tran_q$ and $\Sym_q$ instead of $\Tran(X)$ and $\Sym(X)$, respectively.   

A \emph{finite transformation monoid} is simply a submonoid of $\Tran_q$, for some $q$. This type of monoids has been extensively studied (e.g. see \cite{GM09} and references therein), and it should be noted its close relation to finite-state machines. 

For the rest of the paper, let $G$ be a finite group of size $n$ and $A$ a finite set of size $q$. By Definition \ref{def:ca}, it is clear that $\CA(G;A) \leq \Tran(A^G)$ (we use the symbol ``$\leq$'' for the submonoid relation). We may always assume that $\tau \in \CA(G;A)$ has (not necessarily minimal) memory set $S = G$, so $\tau$ is completely determined by its local function $\mu: A^G \to A$. Hence, $\vert \CA(G ; A) \vert = q^{q^n}$. 

If $n=1$, then $\CA(G;A) = \Tran(A)$, while, if $q \leq 1$, then $\CA(G;A)$ is the trivial monoid with one element; henceforth, we assume $n \geq 2$ and $q \geq 2$. We usually identify $A$ with the set $\{0, 1, \dots, q-1 \}$. 

The group $G$ acts on the configuration space $A^G$ as follows: for each $g \in G$ and $x \in A^G$, the configuration $x \cdot g \in A^G$ is defined by 
\[ (h)x \cdot g = (hg^{-1})x, \quad \forall h \in G. \]
A transformation $\tau : A^G \to A^G$ is \emph{$G$-equivariant} if, for all $x \in A^G$, $g \in G$,
\[ (x \cdot g) \tau = ( (x) \tau ) \cdot g .\] 

Denote by $\ICA(G;A)$ the group of all invertible cellular automata:
\[ \ICA(G;A) := \{ \tau \in \CA(G;A) : \exists \phi \in \CA(G;A) \text{ such that } \tau \phi = \phi \tau = \id \}. \]

\begin{theorem} \label{AG-finite}
Let $G$ be a finite group and $A$ a finite set.
\begin{description}
\item[(i)] $\CA(G;A) = \{ \tau \in \Tran(A^G) : \tau \text{ is $G$-equivariant} \}$.
\item[(ii)]  $\ICA(G;A) = \CA(G;A) \cap \Sym(A^G)$.
\end{description} 
\end{theorem}
\begin{proof}
The first part follows by Curtis-Hedlund Theorem (see \cite[Theorem 1.8.1]{CSC10}) while the second part follows by \cite[Theorem 1.10.2]{CSC10}. \qed
\end{proof}

\begin{notation} \label{notation-orbits}
For any $x \in A^G$, denote by $xG$ the \emph{$G$-orbit} of $x$ on $A^G$:
\[ xG := \{ x \cdot g : g \in G  \}. \]
Let $\mathcal{O}(G;A)$ be the set of all $G$-orbits on $A^G$:
\[ \mathcal{O}(G;A) := \{ xG : x \in A^G\}.\] 
\end{notation}

Clearly, $\mathcal{O}(G;A)$ forms a partition of $A^G$. In general, when $X$ is a set and $\mathcal{P}$ is a partition of $X$, we say that a transformation monoid $M \leq \Tran(X)$ \emph{preserves the partition} if, for any $P \in \mathcal{P}$ and $\tau \in M$ there is $Q \in \mathcal{P}$ such that $(P)\tau \subseteq Q$.

\begin{lemma} \label{preserve}
For any $x \in A^G$ and $\tau \in \CA(G;A)$,
\[ (xG) \tau = (x)\tau G. \]
In particular, $\CA(G;A)$ preserves the partition $\mathcal{O}(G;A)$ of $A^G$.
\end{lemma}
\begin{proof}
The result follows by the $G$-equivariance of $\tau \in \CA(G;A)$.   \qed
\end{proof}

 A configuration $x \in A^G$ is called \emph{constant} if $(g)x = k \in A$, for all $g \in G$. In such case, we usually denote $x$ by $\mathbf{k} \in A^G$. 

\begin{lemma} \label{constant-config}
 Let $\tau \in \CA(G;A)$ and let $\mathbf{k} \in A^G$ be a constant configuration. Then, $(\mathbf{k}) \tau \in A^G$ is a constant configuration.
\end{lemma}
\begin{proof}
Observe that $x \in A^G$ is constant if and only if  $x \cdot g = x$, for all $g \in G$. By $G$-equivariance,
\[ (\mathbf{k}) \tau =  (\mathbf{k} \cdot g) \tau =  (\mathbf{k}) \tau \cdot g,  \quad \forall g \in G. \]
Hence, $(\mathbf{k})\tau$ is constant. \qed 
\end{proof}

For a monoid $M$ and a subset $T \subseteq M$, denote by $C_M(T)$ the \emph{centraliser} of $T$ in $M$:
\[ C_M(T) := \{ m \in M : mt=tm, \forall t \in T \}. \]

If $G$ is abelian, the transformation $\sigma_g : A^G \to A^G$, with $g \in G$, defined by
\[ (x) \sigma_g :=  x\cdot g , \quad \forall x \in A^G, \]
is in $\CA(G;A)$. It follows by Theorem \ref{AG-finite} that $CA(G;A) = C_{\Tran(A^G)}(T)$, where $T :=\{ \sigma_g : g \in G \}$. 

We use the cyclic notation for the permutations of $\Sym(A^G)$. If $B \subseteq A^G$ and $a \in A^G$, we define the idempotent transformation $(B \to a) \in \Tran(A^G)$ by
\[ (x)(B \to a) := \begin{cases}
a & \text{ if } x \in B, \\
x & \text{ otherwise},  
 \end{cases} \quad \forall x \in A^G. \]
When $B=\{ b\}$ is a singleton, we write $(b \to a)$ instead of $(\{ b\} \to a)$.


\section{The Structure of $\ICA(G;A)$} \label{structure}

 Let $G$ be a finite group of size $n \geq 2$ and $A$ a finite set of size $q \geq 2$. We review few basic concepts about permutation groups (see \cite[Ch. 1]{DM96}). For $x \in A^G$, denote by $G_x$ the \emph{stabiliser} of $x$ in $G$:
\[ G_x : = \{g \in G : x \cdot g = x \}.\]

\begin{remark} \label{rk:subgroups}
For any subgroup $H \leq G$ there exists $x \in A^G$ such that $G_x = H$; namely, we may define $x : G \to A$ by
\[ (g)x := \begin{cases}
1 & \text{if } g \in H, \\
0 & \text{otherwise},
\end{cases} \quad \forall g \in G. \]
\end{remark}

Say that two subgroups $H_1$ and $H_2$ of $G$ are \emph{conjugate} in $G$ if there exists $g \in G$ such that $g^{-1} H_1 g = H_2$. This defines an equivalence relation on the subgroups of $G$. Denote by $[H]$ the conjugacy class of $H \leq G$. 

We say that the actions of $G$ on two sets $\Omega$ and $\Gamma$ are \emph{equivalent} if there is a bijection $\lambda : \Omega \to \Gamma$ such that, for all $x \in \Omega, g \in G$, we have $(x \cdot g )\lambda = (x)\lambda \cdot g$.  

The following is an essential result for our description of the structure of the group of invertible cellular automata.

\begin{lemma}\label{conjugate}
Let $G$ be a finite group of size $n \geq 2$ and $A$ a finite set of size $q \geq 2$. For any $x,y \in A^G$, there exists $\tau \in \ICA(G;A)$ such that $(xG)\tau = yG$ if and only if $[G_x] = [ G_y]$.
\end{lemma}
\begin{proof}
By \cite[Lemma 1.6B]{DM96}, the actions of $G$ on $xG$ and $yG$ are equivalent if and only if $G_x$ and $G_y$ are conjugate in $G$. We claim that the actions of $G$ on $xG$ and $yG$ are equivalent if and only if there is $\tau \in \ICA(G;A)$ such that $(xG) \tau = yG$. Assume such $\tau \in \ICA(G;A)$ exists. Then, the restriction $\lambda := \tau \vert_{xG} : xG \to yG$ is the bijection required to show that the actions of $G$ on $xG$ and $yG$ are equivalent. Conversely, suppose there is a bijection $\lambda : xG \to yG$ such that $(z \cdot g )\lambda = (z)\lambda \cdot g$, for all $z \in xG$, $g \in G$. Define $\tau : A^G \to A^G$ by
\[ (z)\tau := \begin{cases}
(z)\lambda & \text{if } z \in xG, \\
(z)\lambda^{-1} & \text{if } z \in yG, \\
z & \text{otherwise},
\end{cases} \quad \forall z \in A^G. \]   
Clearly, $\tau$ is $G$-equivariant and invertible (in fact, $\tau = \tau^{-1}$). Hence $\tau \in \ICA(G;A)$, and it satisfies $(xG)\tau = yG$.  \qed
\end{proof}

\begin{corollary}
Suppose that $G$ is a finite abelian group. For any $x,y \in A^G$, there exists $\tau \in \ICA(G;A)$ such that $(xG)\tau = yG$ if and only if $G_x = G_y$.
\end{corollary}

For any integer $\alpha \geq 2$ and any group $C$, the \emph{wreath product} of $C$ by $\Sym_\alpha$ is the set
\[ C \wr \Sym_{\alpha} := \{ (v; \phi) : v \in C ^\alpha, \phi \in \Sym_\alpha \} \]
equipped with the operation
\[ (v;\phi) \cdot (w; \psi) = ( v w^{\phi}; \phi \psi), \text{ for any } v,w \in C^\alpha, \phi, \psi \in \Sym_\alpha \]
where $\phi$ acts on $w$ by permuting its coordinates:
\[  w^\phi = (w_1, w_2, \dots, w_\alpha)^\phi := (w_{(1)\phi}, w_{(2)\phi}, \dots, w_{(\alpha)\phi}). \]
See \cite[Sec. 2.6]{DM96} for a more detailed description of the wreath product.

\begin{notation} \label{C(GO)-notation}
Let $O \in \mathcal{O}(G;A)$ be a $G$-orbit on $A^G$. If $G_{(O)}$ is the pointwise stabiliser of $O$, i.e. $G_{(O)} := \bigcap_{x \in O} G_x$, then $G^O := G / G_{(O)}$ is a group that is isomorphic to a subgroup of $\Sym(O)$ (see \cite[p. 17]{DM96}). Consider the group
\begin{equation} \label{centraliser}
C(G^O) := \{ \tau \vert_{O} : O \to A^G : \tau \in \ICA(G;A) \text{ and } (O)\tau = O  \}.  
\end{equation}
By Theorem \ref{AG-finite}, $C(G^O)$ is isomorphic to the centraliser of $G^O$ in $\Sym(O)$:
\[ C(G^O) \cong C_{\Sym(O)}(G^O).\]
\end{notation}

\begin{notation}\label{alpha-notation}
Let $H$ be a subgroup of $G$ and $[H]$ its conjugacy class. Define 
\[ B_{[H]} := \{ x \in A^G : G_x \in [H] \}. \]
Note that $B_{[H]}$ is a union of $G$-orbits and, by the Orbit-Stabiliser Theorem (see \cite[Theorem 1.4A]{DM96}), all the $G$-orbits contained in $B_{[H]}$ have equal sizes. Define 
\[ \alpha_{[H]} (G; A) := \left\vert \left\{ O \in \mathcal{O}(G,A) : O \subseteq B_{[H]}   \right\} \right\vert.  \]
If $r$ is the number of different conjugacy classes of subgroups of $G$, observe that 
\[ \mathcal{B} := \{ B_{[H]} : H \leq G \}\]
is a partition of $A^G$ with $r$ blocks. 
\end{notation}

\begin{remark}
$B_{[G]} = \{ x \in A^G : x \text{ is constant} \}$ and $\alpha_{[G]} (G; A) = q$.
\end{remark}

\begin{example} \label{ex:klein}
Let $G = \mathbb{Z}_2 \times \mathbb{Z}_2$ be the Klein four-group and $A= \{ 0, 1\}$. As $G$ is abelian, $[H] = \{ H\}$, for all $H \leq G$. The subgroups of $G$ are 
\[ H_1 = G, \ H_2 = \langle (1,0) \rangle, \ H_3 = \langle (0,1) \rangle, \ H_4 = \langle (1,1) \rangle, \ \text{and} \ H_5 = \langle (0,0) \rangle, \]
where $\langle (a,b) \rangle$ denotes the subgroup generated by $(a,b) \in G$. Any configuration $x : G \to A$ may be written as a $2 \times 2$ matrix $(x_{i,j})$ where $x_{i,j} := (i-1,j-1)x$, $i,j \in \{1,2 \}$. The $G$-orbits on $A^G$ are
\begin{align*}
& O_1  := \left\{  \left( \begin{tabular}{cc}
$0$ \ & \ $0$ \\
 $0$ \  & \ $0$
 \end{tabular} \right) \right\}, \  \ O_2 := \left\{ \left( \begin{tabular}{cc}
$1$ \ & \ $1$ \\
 $1$ \  & \ $1$
 \end{tabular} \right) \right\}, \ \ O_3  :=  \left\{  \left( \begin{tabular}{cc}
$1$ \ & \ $0$ \\
 $1$ \  & \ $0$
 \end{tabular} \right), \left( \begin{tabular}{cc}
$0$ \ & \ $1$ \\
 $0$ \  & \ $1$
 \end{tabular} \right) \right\}, \\[.5em]
 & O_4  :=  \left\{ \left( \begin{tabular}{cc}
$1$ \ & \ $1$ \\
 $0$ \  & \ $0$
 \end{tabular} \right), \left( \begin{tabular}{cc}
$0$ \ & \ $0$ \\
 $1$ \  & \ $1$
 \end{tabular} \right)  \right\}, \  \ O_5  :=  \left\{ \left( \begin{tabular}{cc}
$1$ \ & \ $0$ \\
 $0$ \  & \ $1$
 \end{tabular} \right), \left( \begin{tabular}{cc}
$0$ \ & \ $1$ \\
 $1$ \  & \ $0$
 \end{tabular} \right)   \right\} \\[.5em]
& O_6  :=  \left\{ \left( \begin{tabular}{cc}
$1$ \ & \ $0$ \\
 $0$ \  & \ $0$
 \end{tabular} \right), \left( \begin{tabular}{cc}
$0$ \ & \ $1$ \\
 $0$ \  & \ $0$
 \end{tabular} \right), \left( \begin{tabular}{cc}
$0$ \ & \ $0$ \\
 $0$ \  & \ $1$
 \end{tabular} \right), \left( \begin{tabular}{cc}
$0$ \ & \ $0$ \\
 $1$ \  & \ $0$
 \end{tabular} \right) \right\},  \\[.5em] 
& O_7  :=  \left\{ \left( \begin{tabular}{cc}
$0$ \ & \ $1$ \\
 $1$ \  & \ $1$
 \end{tabular} \right), \left( \begin{tabular}{cc}
$1$ \ & \ $0$ \\
 $1$ \  & \ $1$
 \end{tabular} \right), \left( \begin{tabular}{cc}
$1$ \ & \ $1$ \\
 $1$ \  & \ $0$
 \end{tabular} \right), \left( \begin{tabular}{cc}
$1$ \ & \ $1$ \\
 $0$ \  & \ $1$
 \end{tabular} \right) \right\}.
\end{align*}
Hence,
\begin{align*}
& B_{[H_1]}:=O_1 \cup O_2, \ B_{[H_2]}:=O_3, \ B_{[H_3]}:=O_4, \ B_{[H_4]}:=O_5, \ B_{[H_5]}:=O_6 \cup O_7; \\
& \alpha_{[H_i]}(G;A) = 2, \text { for } i \in \{1,5 \}, \text { and } \alpha_{[H_i]}(G;A) = 1, \text{ for } i \in \{ 2,3,4 \}.
\end{align*}
\end{example}

\begin{remark}
By Lemma \ref{conjugate}, the $\ICA(G;A)$-orbits on $A^G$ coincide with the blocks in $\mathcal{B}$, while the $\ICA(G;A)$-blocks of imprimitivity on each $B_{[H]}$ are the $G$-orbits contained in $B_{[H]}$.
\end{remark}

The following result is a refinement of \cite[Theorem 9]{S15} and \cite[Lemma 4]{CRG15}.

\begin{theorem} \label{th:ICA}
Let $G$ be a finite group and $A$ a finite set of size $q \geq 2$. Let $[H_1], \dots, [H_r]$ be the list of different conjugacy classes of subgroups of $G$. For each $1 \leq i \leq r$, fix a $G$-orbit $O_i \subseteq B_{[H_i]}$. Then,
\[ \ICA(G;A) \cong \prod_{i=1}^r \left( C_i \wr \Sym_{\alpha_i} \right), \]
where $C_i := C(G^{O_i}) \cong C_{\Sym(O_i)}(G^{O_i})$ and $\alpha_i := \alpha_{[H_i]}(G;A)$.
\end{theorem}
\begin{proof}
Let $B_i := B_{[H_i]}$. By Lemma \ref{conjugate}, $\ICA(G;A)$ is contained in the group 
\[  \prod_{i=1}^r \Sym(B_i) = \Sym(B_1) \times \Sym(B_2) \times \dots \times \Sym(B_r). \]
For each $1 \leq i \leq r$, let $\mathcal{O}_i$ be the set of $G$-orbits contained in $B_i$ (so $O_i \in \mathcal{O}_i$). Note that $\mathcal{O}_i$ is a uniform partition of $B_i$. For any $\tau \in \ICA(G;A)$, Lemma \ref{preserve} implies that the projection of $\tau$ to $\Sym(B_i)$ is contained in
\[ S(B_i, \mathcal{O}_i ) := \{ \phi \in \Sym(B_i) : \forall P \in \mathcal{O}_i, \  (P)\phi \in \mathcal{O}_i \}.  \]
By \cite[Lemma 2.1(iv)]{AS09}, 
\[ S(B_i , \mathcal{O}_i ) \cong  \Sym(O_i) \wr \Sym_{\alpha_i}.    \]
It is well-known that $\Sym_{\alpha_i}$ is generated by its transpositions. As the invertible cellular automaton constructed in the proof of Lemma \ref{conjugate} induces a transposition $(xG,yG) \in \Sym_{\alpha_i}$, with $xG, yG \in \mathcal{O}_i$, we deduce that $\Sym_{\alpha_i} \leq \ICA(G;A)$. The result follows by the construction of $C_i \cong C_{\Sym(O_i)}(G^{O_i})$ and Theorem \ref{AG-finite}. \qed  
\end{proof}

\begin{corollary} \label{cor:structure}
Let $G$ be a finite abelian group and $A$ a finite set of size $q\geq 2$. Let $H_1, \dots, H_r$ be the list of different subgroups of $G$. Then,
\[ \ICA(G;A) \cong \prod_{i=1}^r \left( (G/H_i) \wr \Sym_{\alpha_i} \right), \]
and $\vert G \vert \alpha_i  = \vert H_i \vert \cdot \vert \{ x \in A^G : G_x = H_i \} \vert$, where $\alpha_i := \alpha_{[H_i]}(G;A)$.
\end{corollary}
\begin{proof}
By \cite[Theorem 4.2A (v)]{DM96}, $C_{\Sym(O_i)}(G^{O_i}) \cong G^{O_i} \cong G/G_{x_i}$, where $x_i \in O_i$. By Remark \ref{rk:subgroups}, the list of pointwise stabilisers coincide with the list of subgroups of $G$, and, as $G$ is abelian, $[H_i] = \{ H_i \}$ for all $i$. Finally, by the Orbit-Stabiliser theorem, every orbit contained in $B_i = \{ x \in A^G : G_x = H_i \}$ has size $\frac{\vert G \vert}{\vert H_i \vert}$; as these orbits form a partition of $B_i$, we have $\vert B_i \vert = \alpha_i  \frac{\vert G \vert}{\vert H_i \vert}$.  \qed
\end{proof}

\begin{example} \label{ex:ICA-klein}
Let $G = \mathbb{Z}_2 \times \mathbb{Z}_2$ and $A= \{ 0, 1\}$. By Example \ref{ex:klein}, 
\[ \ICA(G, A ) \cong (\mathbb{Z}_2)^4 \times (G \wr \Sym_2). \]
\end{example}


\section{Generating Sets of of $\CA(G;A)$} \label{generating}

For a monoid $M$ and a subset $T \subseteq M$, denote by $\langle T \rangle$ the submonoid \emph{generated} by $T$, i.e. smallest submonoid of $M$ containing $T$. Say that $T$ is a \emph{generating set} of $M$ if $M = \langle T \rangle$; in this case, every element of $M$ is expressible as a word in the elements of $T$ (we use the convention that the empty word is the identity).

Define the \emph{kernel} of a transformation $\tau : X \to X$, denoted by $\ker(\tau)$, as the partition of $X$ induced by the equivalence relation $\{ (x,y ) \in X^2  : (x)\tau = (y) \tau \}$. For example, $\ker(\phi) = \{ \{ x\} : x \in X\}$, for any $\phi \in \Sym(X)$, while $\ker(y \to z) = \{ \{ y, z \}, \{x \} : x \in X \setminus \{ y,z\} \}$, for $y, z \in X$, $y \neq z$.  	 

A large part of the classical research on CA has been focused on CA with small memory sets. In some cases, such as the elementary Rule 110, or John Conway's Game of Life, these CA are known to be Turing complete. In a striking contrast, when $G$ and $A$ are both finite, CA with small memory sets are insufficient to generate the monoid $\CA(G;A)$.

\begin{theorem} \label{minimal-memory}
Let $G$ be a finite group of size $n \geq 2$ and $A$ a finite set of size $q \geq 2$. Let $T$ be a generating set of $\CA(G;A)$. Then, there exists $\tau \in T$ with minimal memory set $S=G$.
\end{theorem}
\begin{proof}
Suppose that $T$ is a generating set of $\CA(G, A)$ such that each of its elements has minimal memory set of size at most $n-1$. Consider the idempotent $\sigma:=(\mathbf{0} \to \mathbf{1}) \in \CA(G, A)$, where $\textbf{0}, \textbf{1} \in A^G$ are different constant configurations. Then, $\sigma = \tau_1 \tau_2 \dots \tau_\ell$, for some $\tau_i \in T$. By the definition of $\sigma$, there must be $1 \leq j \leq \ell$ such that $\ker(\tau_j) = \{ \{ \mathbf{0}, \mathbf{1} \}, \{ x \} : x \in A^G\setminus \{ \textbf{0}, \textbf{1}\} \}$. By Lemma \ref{constant-config}, $( A^G_{\text{c}})\tau_j \subseteq A^G_{\text{c}}$ and $( A^G_{\text{nc}})\tau_j = A^G_{\text{nc}}$, where 
\[A^G_{\text{c}}:= \{ \mathbf{k} \in A^G : \mathbf{k} \text{ is constant} \} \text{ and } A^G_{\text{nc}} := \{ x \in A^G : x \text{ is non-constant} \}. \]
 Let $S \subseteq G$ and $\mu : A^S \to A$ be the minimal memory set and local function of $\tau := \tau_j$, respectively. By hypothesis, $s := \vert S \vert < n$. Since the restriction of $\tau$ to $A^G_{\text{c}}$ is not a bijection, there exists $\mathbf{k} \in A^G_{\text{c}}$ (defined by $(g)\mathbf{k}:=k \in A$, $\forall g \in G$) such that $\mathbf{k} \not \in ( A^G_{\text{c}})\tau$. 

For any $x \in A^G$, define the $k$-\emph{weight} of $x$ by
\[   \vert x \vert_k := \vert \{ g \in G : (g)x \neq k \} \vert. \] 
Consider the sum of the $k$-weights of all non-constant configurations of $A^G$:
\[ w := \sum_{x \in A^G_{\text{nc}} } \vert x \vert_k  = n(q-1) q^{n-1}  - n(q-1) =  n(q-1) ( q^{n-1} - 1) .  \]
In particular, $\frac{w}{n}$ is an integer not divisible by $q$.

For any $x \in A^G$ and $y \in A^S$, define 
\[ \Sub(y, x) := \vert \{ g \in G : y = x \vert_{Sg} \} \vert. \]
Then, for any $y \in A^S$,
\[ N_y  := \sum_{x \in A^G_{\text{nc}}} \Sub(y,x) = \begin{cases}
 n q^{n-s}& \text{if } y \in A^S_{\text{nc}}, \\
 n (q^{n-s} - 1 )& \text{if } y \in A^S_{\text{c}}. 
\end{cases} \]

Let $\delta : A^2 \to \{0,1 \}$ be the Kronecker's delta function. Since $( A^G_{\text{nc}})\tau = A^G_{\text{nc}}$, we have
\begin{align*}
 w &= \sum_{x \in A^G_{\text{nc}} } \vert (x)\tau \vert_k = \sum_{y \in A^S} N_y  ( 1 - \delta( (y)\mu, k)  )    \\
& =   n q^{n-s} \sum_{y \in A^S_{\text{nc}}} ( 1 - \delta( (y)\mu, k) ) +  n ( q^{n-s} - 1 ) \sum_{y \in A^S_{\text{c}}} ( 1 - \delta( (y)\mu, k) ). 
\end{align*}
Because $\mathbf{k} \not \in ( A^G_{\text{c}})\tau$, we know that $(y)\mu \neq k$ for all $y \in A^S_{\text{c}}$. Therefore,
\[ \frac{w}{n} =  q^{n-s} \sum_{y \in A^S_{\text{nc}}} ( 1 - \delta_{ (y)\mu, k} ) +  ( q^{n-s} - 1 ) q.  \]
As $s <n$, this implies that $\frac{w}{n}$ is an integer divisible by $q$, which is a contradiction. \qed
\end{proof}

One of the fundamental problems in the study of a finite monoid $M$ is the determination of the cardinality of a smallest generating subset of $M$; this is called the \emph{rank} of $M$ and denoted by $\Rank(M)$:
\[ \Rank(M) := \min \{ \vert T \vert :  T \subseteq M \text{ and } \langle T \rangle = M \}. \]
It is well-known that, if $X$ is any finite set, the rank of the full transformation monoid $\Tran(X)$ is $3$, while the rank of the symmetric group $\Sym(X)$ is $2$ (see \cite[Ch.~3]{GM09}). Ranks of various finite monoids have been determined in the literature before (e.g. see \cite{ABJS14,AS09,GH87,G14,HM90}). 

In \cite{CRG15}, the rank of $\CA(\mathbb{Z}_n, A)$, where $\mathbb{Z}_n$ is the cyclic group of order $n$, was studied and determined when $n \in \{ p, 2^k, 2^k p : k \geq 1, \ p \text{ odd prime} \}$. Moreover, the following problem was proposed:
\begin{problem}\label{problem}
For any finite group $G$ and finite set $A$, determine $\Rank(\CA(G;A))$.
\end{problem}                      

For any finite monoid $M$ and $U \subseteq M$, the \emph{relative rank} of $U$ in $M$, denoted by $\Rank(M:U)$, is the minimum cardinality of a subset $V \subseteq M$ such that $\langle U \cup V \rangle = M$. For example, for any finite set $X$, 
\[ \Rank(\Tran(X): \Sym(X)) = 1, \]
as any $\tau \in \Tran(X)$ with $\vert (X) \tau \vert = \vert X \vert -1$ satisfies $\langle \Sym(X) \cup \{ \tau \} \rangle = \Tran(X)$. One of the main tools that may be used to determine $\Rank(\CA(G;A))$ is based on the following result (see \cite[Lemma 3.1]{AS09}).

\begin{lemma} \label{le:preliminar}
Let $G$ be a finite group and $A$ a finite set. Then,
\[ \Rank(\CA(G ; A)) = \Rank(\CA(G;A):\ICA(G;A)) + \Rank(\ICA(G;A)). \]
\end{lemma}

We shall determine the relative rank of $\ICA(G;A)$ in $\CA(G;A)$ for any finite abelian group $G$ and finite set $A$. In order to achieve this, we prove two lemmas that hold even when $G$ is nonabelian and have relevance in their own right. 

\begin{lemma} \label{le:action-orbit}
 Let $G$ be a finite group and $A$ a finite set of size $q\geq 2$. Let $x\in A^G$. If $(xG)\tau = xG$, then $\tau \vert_{xG} \in \Sym(xG)$. 
\end{lemma} 
\begin{proof}
It is enough to show that $\tau \vert_{xG} : xG \to xG$ is surjective because $xG$ is finite. Let $y \in xG$. Since $(x)\tau \in xG$, there is $g \in G$ such that $y = (x)\tau \cdot g $. By $G$-equivariance, $y = (x \cdot g) \tau \in (xG)\tau$, and the result follows. \qed  
\end{proof}

\begin{notation}
Denote by $\mathcal{C}_G$ the set of conjugacy classes of subgroups of $G$. For any $[H_1], [H_2] \in \mathcal{C}_G$, write $[H_1] \leq [H_2]$ if $H_1 \leq g^{-1} H_2 g$, for some $g \in G$.
\end{notation}

\begin{remark}
The relation $\leq$ defined above is a well-defined partial order on $\mathcal{C}_G$. Clearly, $\leq$ is reflexive and transitive. In order to show antisymmetry, suppose that $[H_1] \leq [H_2]$ and $[H_2] \leq [H_1]$. Then, $H_1 \leq g^{-1} H_2 g$ and $H_2 \leq f^{-1} H_2 f$, for some $f,g \in G$, which implies that $\vert H_1 \vert \leq \vert H_2 \vert$ and $\vert H_2 \vert \leq \vert H_1 \vert$. As $H_1$ and $H_2$ are finite, $\vert H_1 \vert = \vert H_2 \vert$, and $H_1 = g^{-1} H_2 g$. This shows that $[H_1] = [H_2]$.     
\end{remark}

\begin{lemma} \label{le:idem}
Let $G$ be a finite group and $A$ a finite set of size $q\geq 2$. Let $x , y \in A^G$ be such that $xG \neq yG$. There exists a non-invertible $\tau \in \CA(G;A)$ such that $(xG)\tau = yG$ if and only if $[G_x] \leq [G_y]$.
\end{lemma}
\begin{proof}
Suppose that $[G_x] \leq [G_y]$. Then, $G_x \leq g^{-1} G_y g$, for some $g \in G$. We define an idempotent $\tau_{x,y} : A^G \to A^G$ that maps $xG$ to $yG$:
\[ (z) \tau_{x,y} := \begin{cases}
y \cdot g h & \text{if } z = x \cdot h, \\ 
z & \text{otherwise},
 \end{cases}  \quad \forall z \in A^G. \] 
We verify that $\tau_{x,y}$ is well-defined. If $x \cdot h_1 = x \cdot  h_2$, for $h_i \in G$, then $h_1 h_2^{-1} \in G_x$. As $G_x \leq g^{-1} G_y g$, for some $s \in G_y$, we have $ h_1 h_2^{-1} = g^{-1} s g$. Thus, $g h_1 = s g h_2$ implies that $y \cdot gh_1 = y \cdot g h_2$, and $(x \cdot h_1 )\tau = (x \cdot h_2) \tau$. Clearly, $\tau_{x,y}$ is non-invertible and $G$-equivariant, so $\tau_{x,y}\in \CA(G;A)$.

Conversely, suppose there exists $\tau \in \CA(G;A)$ such that $(xG)\tau = yG$. Then, $(x)\tau = y \cdot h$, for some $h \in G$. Let $s \in G_x$. By $G$-equivariance, 
\[ y\cdot h = (x) \tau = (x \cdot s) \tau = (x) \tau \cdot s = y \cdot hs .\]
Thus $h s h^{-1} \in G_y$ and $s \in h^{-1} G_y h$. This shows that $[G_x] \leq [G_y]$. \qed
\end{proof}

\begin{corollary} \label{cor:idem}
Suppose that $G$ is finite abelian. Let $x,y \in A^G$ be such that $xG \neq yG$. There exists $\tau_{x,y} \in \CA(G;A)$ such that $(x)\tau_{x,y} = y$ and $(z)\tau_{x,y} = z$ for all $z \in A^G \setminus xG$ if and only if $G_x \leq G_y$. 
\end{corollary}

\begin{notation} \label{notation-edges}
Consider the directed graph $(\mathcal{C}_G, \mathcal{E}_G)$ with vertex set $\mathcal{C}_G$ and edge set
\[ \mathcal{E}_G := \left\{ ([H_i], [H_j]) \in \mathcal{C}_G^2 : [H_i] \leq [H_j] \right\}.  \]
When $G$ is abelian, this graph coincides with the lattice of subgroups of $G$.
\end{notation}

\begin{remark}
Lemma \ref{le:idem} may be restated in terms of $\mathcal{E}_G$. By Lemma \ref{le:action-orbit}, loops $([H_i],[H_i])$ do not have corresponding non-invertible CA when $\alpha_{[H_i]}(G;A)=1$.  
\end{remark}

\begin{theorem} \label{th:relative rank}
Let $G$ be a finite abelian group and $A$ a finite set of size $q\geq 2$. Let $H_1, H_2, \dots, H_r$ be the list of different subgroups of $G$ with $H_1 = G$. For each $1 \leq i \leq r$, let $\alpha_i := \alpha_{[H_i]}(G;A)$. Then,
\[ \Rank(\CA(G;A):\ICA(G;A)) =  \vert \mathcal{E}_G \vert - \sum_{i=2}^r \delta(\alpha_i,1), \]
where $\delta : \mathbb{N}^2 \to \{0,1 \}$ is Kronecker's delta function. 
\end{theorem}
\begin{proof}
For all $1 \leq i \leq r$, let $B_i := B_{[H_i]}$. Fix orbits $x_i G \subseteq B_i$, so $H_i = G_{x_i}$. Assume that the list of subgroups of $G$ is ordered such that
 \[ \vert x_1 G \vert \leq \dots \leq \vert x_r G \vert, \text{ or, equivalently, } \vert G_{x_1} \vert  \geq \dots \geq \vert G_{x_r} \vert . \]  
For every $\alpha_i \geq 2$, fix orbits $y_i G \subseteq B_i$ such that $x_i G \neq y_i G$. We claim that $\CA(G,A) = M:= \left\langle \ICA(G;A) \cup U \right\rangle$, where
\[ U := \left\{ \tau_{x_i,x_j} : [G_{x_i}] < [G_{x_j}] \right\} \cup \left\{ \tau_{x_i, y_i} : \alpha_i \geq 2 \right\}, \]
and $\tau_{x_i, x_j}, \tau_{x_i, y_i}$ are the idempotents defined in Corollary \ref{cor:idem}. For any $\tau \in \CA(G;A)$, consider $\tau_i \in \CA(G;A)$, $1 \leq i \leq r$, defined by
\[ (x)\tau_i  = \begin{cases}
(x)\tau & \text{if } x \in B_i \\
x & \text{otheriwse}.
\end{cases}\]
By Lemmas \ref{conjugate} and \ref{le:idem}, $(B_i)\tau \subseteq \bigcup_{j \leq i} B_j$ for all $i$. Hence, we have the decomposition
\[ \tau = \tau_1 \tau_2 \dots \tau_r.  \] 
For each $i$, decompose $\tau_i$ further as $\tau_i = \tau_i^{\prime} \tau_{i}^{\prime \prime}$, where $(B_i)\tau_i^{\prime} \subseteq \bigcup_{j < i} B_j$ and $(B_i)\tau_{i}^{\prime \prime} \subseteq B_i$. We shall prove that $\tau_i^\prime \in M$ and $\tau_i^{\prime \prime} \in M$.
\begin{enumerate}
\item We show that $\tau_i^{\prime} \in M$. If $B_i = \cup_{s=1}^{\alpha_i} P_s$ is the decomposition of $B_i$ into its $G$-orbits, we may write $\tau_i^\prime = \tau_i^\prime \vert_{P_1} \dots \tau_i^\prime \vert_{P_{\alpha_i}}$, where $\tau_i^\prime \vert_{P_s}$ acts as $\tau_i^\prime$ on $P_s$ and fixes everything else. Note that $Q_s = (P_s)\tau_i^\prime \vert_{P_s}$ is a $G$-orbit in $B_j$ for some $j<i$. By Theorem \ref{th:ICA}, there exist 
\[ \phi_s \in \left( (G/ G_{x_i}) \wr \Sym_{\alpha_i} \right) \times \left( (G/ G_{x_j}) \wr \Sym_{\alpha_j} \right) \leq \ICA(G;A) \]
such that $\phi_s$ acts as the double transposition $(x_i G, P_s) (x_j G,  Q_s)$. Since $G/ G_{x_i}$ and $G/ G_{x_j}$ are transitive on their respective orbits, we may take $\phi_s$ such that $(x_i) \phi_s \tau_i^\prime \vert_{P_s}\phi_s^{-1} = x_j$. Then,
\[  \tau_i^\prime \vert_{P_s} = \phi_s^{-1}  \tau_{x_i, x_j} \phi_s \in M. \]

\item We show $\tau_{i}^{\prime \prime} \in M$. In this case, $\tau_i^{\prime \prime} \in \Tran(B_i)$. In fact, as $\tau_i^{\prime \prime}$ preserves the partition of $B_i$ into $G$-orbits, Lemma \ref{le:action-orbit} implies that $\tau_i^{\prime \prime} \in (G/G_{x_i}) \wr \Tran_{\alpha_i}$. If $\alpha_i \geq 2$, the semigroup $\Tran_{\alpha_i}$ is generated by $\Sym_{\alpha_i} \leq \ICA(G,A)$ togerher with the idempotent $\tau_{x_i, y_i}$. Hence, $\tau_i^{\prime \prime} \in M$.    
\end{enumerate}
Therefore, we have established that $\CA(G;A) = \left\langle \ICA(G;A) \cup U \right\rangle$.

Suppose now that there exists $V \subseteq \CA(G;A)$ such that $\vert V \vert < \vert U \vert$ and 
\[ \left\langle \ICA(G ; A) \cup V  \right\rangle = \CA(G ; A). \]
Hence, for some $\tau \in U$, we must have 
\[  V \cap \langle \ICA(G; A) , \tau \rangle = \emptyset. \]
If $\tau = \tau_{x_i,y_i}$, for some $i$ with $\alpha_i \geq 2$, this implies that there is no $\xi \in V$ with
\[ \ker(\xi) =  \left\{ \{a, b \}, \{ c \} : a \in x_i G, \ b \in y_i G, \ c \in A^G \setminus (x_iG \cup y_i G) \right\}. \]
Hence, there is no $\xi \in \left\langle \ICA(G ; A) \cup V  \right\rangle = \CA(G ; A)$ with kernel of this form, which is a contradiction because $\tau_{x_i,y_i}$ itself has kernel of this form. We obtain a similar contradiction if $\tau = \tau_{x_i,x_j}$ with $[G_{x_i}] < [G_{x_j}]$.   \qed
\end{proof}

\begin{corollary} \label{cor:bound}
Let $G$ be a finite abelian group with $\Rank(G) = m$ and $A$ a finite set of size $q \geq 2$. With the notation of Theorem \ref{th:relative rank},
\begin{align*}
\Rank(\CA(G;A)) & \leq  \sum_{i=2}^r  m \alpha_i + 2 r + \vert \mathcal{E}_G \vert - \delta(q,2) - \sum_{i=2}^r ( 3\delta(\alpha_i,1) + \delta(\alpha_i,2) )   \\
&  \leq  \sum_{i=2}^r  m \alpha_i + 2r  +  r^2. 
\end{align*}
\end{corollary}
\begin{proof}
Using the fact $\Rank((G/H_i)\wr \Sym_{\alpha_i}) \leq m \alpha_i + 2 - 2\delta(\alpha_i,1) - \delta(\alpha_i,2)$ and $\Rank((G/H_1)\wr \Sym_{q}) = 2 - \delta(q,2)$, the result follows by Theorem \ref{th:relative rank}, Corollary \ref{cor:structure} and Lemma \ref{le:preliminar}.  \qed
\end{proof}

The bound of Corollary \ref{cor:bound} may become tighter if we actually know $\Rank(G/H_i)$, for all $H_i \leq G$, as in Example \ref{ex:ICA-klein}.  

\begin{example}
Let $G=\mathbb{Z}_2 \times \mathbb{Z}_2$ be the Klein-four group and $A = \{ 0,1 \}$. With the notation of Example \ref{ex:klein}, Figure \ref{Fig1} illustrates the Hasse diagram of the subgroup lattice of $G$ (i.e. the actual lattice of subgroups is the transitive and reflexive closure of this graph). 
\begin{figure}[h]
\centering
\begin{tikzpicture}[vertex/.style={circle, draw, fill=none, inner sep=0.55cm}]
    \vertex{1}{1}{2.2}    \node at (1,3.7) {$H_5 \cong \mathbb{Z}_1$};  
    \vertex{2}{3}{1}    \node at (3,1.7) {$H_4 \cong \mathbb{Z}_2$};   
   \vertex{3}{1}{1}    \node at (1,1.7) {$H_3 \cong \mathbb{Z}_2$};  
   \vertex{4}{-1}{1}    \node at (-1,1.7) {$H_2 \cong \mathbb{Z}_2$};  
   \vertex{5}{1}{-.2}   \node at (1,-.4) {$H_1 = G$};   

  \arc{1}{2}
  \arc{1}{3}
 \arc{1}{4}
\arc{2}{5}
\arc{3}{5}
\arc{4}{5}
   \end{tikzpicture} 
\caption{Lattice of subgroups of $G = \mathbb{Z}_2 \times \mathbb{Z}_2$.}
\label{Fig1}
   \end{figure}
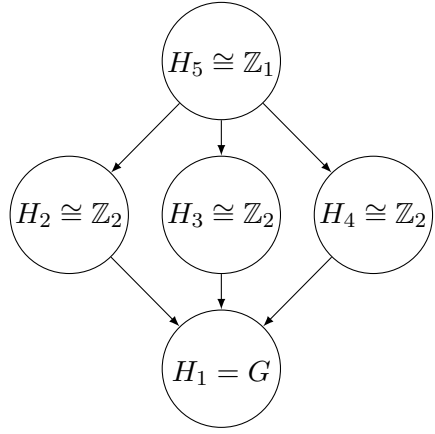

Hence, by Theorem \ref{th:relative rank} and Example \ref{ex:ICA-klein},
\begin{align*}
& \Rank(\CA(G;A):\ICA(G;A))  =  \vert \mathcal{E}_G \vert  - 3 = 12 - 3 = 9,\\
& \Rank(\CA(G;A))   \leq 9 + 9 = 18, \text{ as } \Rank(\ICA(G;A)) \leq 9.  
\end{align*}
\end{example}

Because of Theorem \ref{th:relative rank}, it is particularly relevant to determine in which situations $\alpha_{[H]}(G;A)=1$. We finish this paper with some partial results in this direction that hold for arbitrary finite groups.

Denote by $[G:H]$ the index of $H \leq G$ (i.e. the number of cosets of $H$ in $G$).

\begin{lemma}
Let $G$ be a finite group and $A$ a finite set of size $q\geq 2$. Assume there is $H \leq G$ with $[G:H] = 2$. Then, $\alpha_{[H]} (G;A) = 1$ if and only if $q = 2$.
\end{lemma}
\begin{proof}
As $H \leq G$ has index $2$, it is normal. Fix $s \in G \setminus H$. Define $x \in A^G$ by
\[ (g)x = \begin{cases}
0 & \text{if } g \in H \\
1 & \text{if } g \in sH = Hs.
\end{cases} \]
Clearly $G_x = H$ and $x \in B_{[H]}$.

Suppose first that $A = \{0,1 \}$. Let $y \in B_{[H]}$. As $H$ is normal, $[H] = \{ H\}$, so $G_y = H$. For any $h \in H$,
\[ (h)y = (e)y \cdot h^{-1} = (e)y \text{ and } (sh) y = (s) y \cdot h^{-1} = (s)y,  \]
so $y$ is constant on the cosets $H$ and $sH = Hs$. Therefore, either $y = x$, or
\[ (g)y = \begin{cases}
1 & \text{if } g \in H \\
0 & \text{if } g \in sH = Hs.
\end{cases} \]     
In the latter case, $y \cdot s = x$ and $y \in xG$. This shows that there is a unique $G$-orbit contained in $B_{[H]}$, so $\alpha_{[H]}(G;A) = 1$. 

If $\vert A \vert \geq 3$, we may use a similar argument as above, except that now $y \in B_{[H]}$ may satisfy $(g)y \in A \setminus \{ 0,1\}$ for all $g \in H$, so $y \not \in xG$ and $\alpha_{[H]}(G;A) \geq 2$.\qed
\end{proof}

\begin{lemma}
Let $G$ be a finite group and $A$ a finite set of size $q\geq 2$. Suppose there is $H \leq G$ such that $\alpha_{[H]}(G;A) = 1$. Then, $q \mid [G:H] = \frac{\vert G \vert}{\vert H \vert}$.
\end{lemma}
\begin{proof}
Let $x \in B_{[H]}$ be such that $G_x = H$. As $\alpha_{[H]}(G;A) = 1$, $B_{[H]}= xG$. First we show that $x : G \to A$ is surjective. If $(G)x \subset A$, let $a \in (G)x$ and $b \in A \setminus (G)x$. Define $y \in A^G$ by
\[ (g)y := \begin{cases}
b & \text{ if } (g)x = a \\
(g)x & \text{ otherwise.}
\end{cases}\]
Then $y \in B_{[H]}$, as $G_y = G_x$, but $y \not \in xG$, which is a contradiction. For $a \in A$, let $(a) x^{-1} := \{ g \in G : (g)x = a \}$. Now we show that, for any $a, b \in A$,
\[ \vert (a)x^{-1} \vert = \vert (b)x^{-1} \vert. \]  
Suppose that $\vert (a)x^{-1} \vert < \vert (b)x^{-1} \vert$. Define $z \in A^G$ by
\[ (g)z := \begin{cases}
b & \text{ if } (g)x = a \\
a & \text{ if } (g)x = b \\
(g)x & \text{ otherwise.}
\end{cases}\]
Again, $z \in  B_{[H]}$, as $G_z = G_x$, but $z \not \in xG$, which is a contradiction. As $x$ is constant on the left cosets of $H$ in $G$, for each $a \in A$, $(a)x^{-1}$ is a union of left cosets. All cosets have the same size, so $(a)x^{-1}$ and $(b)x^{-1}$ contain the same number of them, for any $a,b \in A$. Therefore, $q \mid [G:H ]$. \qed      
\end{proof}
\begin{corollary}
Let $G$ be a finite abelian group and $A$ a finite set of size $q\geq 2$ such that $q \nmid \vert G \vert$. With the notation of Theorem \ref{th:relative rank},
\[\Rank(\CA(G;A):\ICA(G;A)) =  \vert \mathcal{E}_G \vert.  \]
\end{corollary}

\subsubsection*{Acknowledgments.} This work was supported by the EPSRC grant EP/K033956/1.


\begin{thebibliography}{4}

\bibitem{ABJS14} Ara\'ujo, J., Bentz, W., Mitchell, J.D., Schneider, C.: The rank of the semigroup of transformations stabilising a partition of a finite set. Mat. Proc. Camb. Phil. Soc. \textbf{159}, 339--353 (2015).

\bibitem{AS09} Ara\'ujo, J., Schneider, C.: The rank of the endomorphism monoid of a uniform partition. Semigroup Forum \textbf{78}, 498--510 (2009).

\bibitem{CRG15} Castillo-Ramirez, A., Gadouleau, M.: Ranks of finite semigroups of one-dimensional cellular automata, \url{http://arxiv.org/abs/1510.00197} (2015).

\bibitem{CSC10} Ceccherini-Silberstein, T., Coornaert, M.: Cellular Automata and Groups. Springer Monographs in Mathematics, Springer-Verlag Berlin Heidelberg (2010).

\bibitem{DM96} Dixon, J.D., Mortimer, B.: Permutation Groups. Graduate Texts in Mathematics \textbf{163}, Springer-Verlag, New York (1996).

\bibitem{GM09} Ganyushkin, O., Mazorchuk, V.: Classical Finite Transformation Semigroups: An Introduction. Algebra and Applications 9, Springer-Verlag, London (2009).

\bibitem{GH87} Gomes, G.M.S., Howie, J.M.: On the ranks of certain finite semigroups of transformations. Math. Proc. Camb. Phil. Soc. \textbf{101}, 395--403 (1987).

\bibitem{G14} Gray, R.D.: The minimal number of generators of a finite semigroup. Semigroup Forum \textbf{89}, 135--154 (2014).

\bibitem{H12} Hartman, Y.: Large semigroups of cellular automata. Ergodic Theory Dyn. Syst. \textbf{32}, 1991--2010 (2012).

\bibitem{HM90} Howie, J.M., McFadden, R.B.: Idempotent rank in finite full transformation semigroups. Proc. Royal Soc. Edinburgh \textbf{114A}, 161--167 (1990).

\bibitem{Ka05} Kari, J.: Theory of cellular automata: A Survey. Theoret. Comput. Sci. \textbf{334}, 3--33 (2005).

\bibitem{S15} Salo, V.: Groups and Monoids of Cellular Automata. In: Kari, J. (ed.) Cellular Automata and Discrete Complex Systems. LNCS, vol. 9099, pp. 17--45, Springer Berlin Heidelberg (2015).
\end{thebibliography}
\end{document}